\documentclass[12pt]{article}
\usepackage{amssymb,latexsym,theorem}

\newcommand{\Spec}{\mathop{\mathrm{Spec}}\nolimits}

\newcommand{\hull}{\mathop{\mathrm{hull}}\nolimits}
\newcommand{\even}{{\mathrm{even}}}
\newcommand{\hgt}{\mathop{\mathrm{ht}}\nolimits}
\newcommand{\gr}{\mathop{\mathrm{gr}}\nolimits}

\newcommand{\gl}{\mathop{\mathfrak{gl}}\nolimits}

\newcommand{\SL}{\mathop{\mathit{SL}}\nolimits}

\newcommand{\Hom}{\mathop{\mathrm{Hom}}\nolimits}
\newcommand{\Z}{\mathbb Z}

\newcommand{\Pp}{\mathbb P}
\newcommand{\Ga}{{\mathbb G_a}}

\newcommand{\cO}{\mathcal O}

\newcommand{\qed}{\unskip\nobreak\hfill\hbox{ $\Box$}}
\newcommand{\mod}{\bmod}

\def\k{{\mathbf k}}

\newtheorem{Theorem}[subsection]{Theorem}
\newtheorem{Lemma}[subsection]{Lemma}

\theorembodyfont{\normalfont}
\newtheorem{Remark}[subsection]{Remark}

\newtheorem{Definition}[subsection]{Definition}
\newtheorem{Notation}[subsection]{Notation}

\newenvironment{proof}{\emph{Proof}}{\qed\\}
\begin{document}
\title{Cohomological finite generation for the group scheme $SL_2$.}
\author{Wilberd van der Kallen}
\date{\emph{Dedicated to the memory of T. A. Springer}}
\maketitle
\sloppy
\begin{abstract}
Let $G$ be the group scheme $\SL_2$ defined over a noetherian ring $\k$.
If $G$ acts on a finitely generated commutative $\k$-algebra $A$, then $H^*(G,A)$
is a finitely generated $\k$-algebra.
\end{abstract}

\section{Introduction}
Let $\k$ be a noetherian ring. 
Consider  a flat linear algebraic  group scheme
$G$ defined over  $\k$.
Recall that $G$ has the cohomological finite generation property (CFG)
if the following holds:
Let $A$ be a finitely generated commutative $\k$-algebra on which $G$ acts
 rationally by $\k$-algebra automorphisms. (So $G$ acts from the right
 on $\Spec(A)$.)
Then the cohomology ring $H^*(G,A)$ is finitely generated
as a $\k$-algebra. Here, as in
\cite[I.4]{J}, we use the cohomology introduced by
Hochschild, also known as `rational cohomology'.

This note is part of the project of studying (CFG) for reductive $G$.
Recall that the breakthrough of Touz\'e \cite{Touze} settled the case when $\k$ is a field \cite{TvdK}.
And \cite[Theorem 10.1]{goodfamily} extended this to the case that $\k$ contains a field. 
In this paper we show that in the case $G=\SL_2$ one can dispense with the condition that $\k$
contains a field. According to the last item of \cite[Theorem 10.5]{goodfamily} it suffices to show that 
$H^*(G,A/pA)$ is a noetherian module over $H^*(G,A)$ whenever $p$ is a prime number.
We fix $p$. To prove the noetherian property we employ universal cohomology classes as in earlier work. 
More specifically, we
 lift the cohomology classes $c_r[a]^{(j)}$ of \cite[4.6]{vdK coh} to classes
in cohomology of $\SL_2$ over the integers 
 with flat coefficient module $\Gamma^{m}\Gamma^{p^{r+j}}(\gl_2)$. 
We get the lifts with explicit formulas that do not seem to generalize to $\SL_n$ with $n>2$. 
Once we have the lifts of the cohomology classes we can lift enough of the mod $p$ constructions to conclude
that $H^*(G,A)$ hits much of $H^*(G,A/pA)$. As $H^*(G,A/pA)$ itself is a finitely generated $\k$-algebra this 
will then imply that $H^*(G,A/pA)$ is a noetherian module over $H^*(G,A)$.

For simplicity of reference we use \cite{goodfamily}. As we are working with $\SL_2$ that amounts to serious overkill.
For instance, the work of Touz\'e is not needed for $\SL_2$. 
Further the `functorial resolution of the ideal of the
diagonal in a product of Grassmannians'  now just means that the
ideal sheaf of the diagonal divisor in a product of two projective lines is the familiar line bundle
$\cO(-1)\boxtimes \cO(-1)$. And
Kempf vanishing for $\SL_2$ is immediate from the computation of the cohomology of line bundles on $\Pp^1$.

\section{Rank one}
We take $G=\SL_2$ as group scheme over the noetherian ring $\k$. 
Initially $\k$ is just $\Z$.
Let $T$ be the diagonal torus and $B$ the Borel subgroup of
lower triangular matrices. Its root $\alpha$ is the negative root.

\subsection{Cocycles for the additive group.}
We have fixed a prime $p$.
Define $\Phi(X,Y)\in \Z[X,Y]$ by $$(X+Y)^p=X^p+Y^p+p\Phi(X,Y).$$

By induction one gets for $r\geq1$
$$(X+Y)^{p^r}\equiv X^{p^r}+Y^{p^r}+p\Phi(X^{p^{r-1}},Y^{p^{r-1}})\bmod p^2.$$
Put $$c_r^\Z(X,Y)=\frac{(X+Y)^{p^r}-X^{p^r}-
Y^{p^r}}p\in \Z[X,Y].$$
We think of $c_r^\Z$ as a 2-cochain in the Hochschild complex $C^\bullet(\Ga,\Z)$ as treated in \cite[I 4.14, I 4.20]{J}.
Then $c_r^\Z$ is a $2$-cocycle  because $pc_r^\Z$ is a coboundary.
One has $$c_r^\Z(X,Y)\equiv\Phi(X^{p^{r-1}},Y^{p^{r-1}})\bmod p.$$
Taking cup products one finds a $2m$-cocycle $c_r^\Z(X,Y)^{\cup m}$ representing a class in $H^{2m}(\Ga,\Z)$.
The cocycle $c_r^\Z$ serves as lift of the $(r-1)$-st Frobenius twist of the Witt vector class that was our starting point in 
\cite[\S 4]{vdK coh}. We can now follow \cite[\S 4]{vdK coh}, lifting all relevant mod $p$ constructions to the integers.
That will do the trick.

\subsection{Universal classes}
Our next task is to construct a universal 
class $c_r[m]^{(j)}$  in $H^{2mp^{r-1}}(G,\Gamma^{m}\Gamma^{p^{r+j}}(\gl_2))$.

Let $r\geq1$, $j\geq0$, $m\geq1$. 
Let $\alpha$ be the negative root, and let $x_\alpha :\Ga\to \SL_2$ be
its  root homomorphism, with image $U_\alpha$. For a $\Z$-module $V$ its $m$-th module of divided powers is written
$\Gamma^m V$ and its dual $\Hom_\Z(V,\Z)$ is written $V^\#$.

Consider the representation $\Gamma^{mp^{r+j}}(\gl_2)$ of $G$ with its restriction 
$x_\alpha^*\Gamma^{mp^{r+j}}(\gl_2)$ to $\Ga$. Its lowest weight is $mp^{r+j}\alpha$. 
Say $e_\alpha$ is the elementary matrix $\pmatrix{0&0\cr 1&0}$ that spans the $\alpha$ weight space of $\gl_2$, and $e_\alpha^{[mp^{r+j}]}$
denotes its divided power in $\Gamma^{mp^{r+j}}(\gl_2)$. 
Then $c_{j+1}^\Z(X,Y)^{\cup mp^{r-1}}e_\alpha^{[mp^{r+j}]}$ represents a class in
$H^{2mp^{r-1}}(\Ga,x_\alpha^*\Gamma^{mp^{r+j}}(\gl_2))$ and the corresponding element of  
$H^{2mp^{r-1}}(U_\alpha,\Gamma^{mp^{r+j}}(\gl_2))$ 
is $T$-invariant.
So we get a class in $H^{2mp^{r-1}}(B,\Gamma^{mp^{r+j}}(\gl_2))$ and by Kempf vanishing (\cite[II B.3]{J} with $\lambda=0$) 
a class in $H^{2mp^{r-1}}(G,\Gamma^{mp^{r+j}}(\gl_2))$. 
Recall that one obtains a natural map from $\Gamma^{p^{r+j}}({\gl_2}\mod p)$ to the $(r+j)$-th Frobenius twist
$({\gl_2}\mod p)^{({r+j})}$
by dualizing the  map from $({\gl_2^\#}\mod p)$ to 
$S^{p^{r+j}}({\gl_2^\#}\mod p)$ that raises a vector $v\in({\gl_2^\#}\mod p)$ to its $p^{{r+j}}$-th power.
So $\Gamma^{mp^{r+j}}(\gl_2)$ maps naturally to 
$\Gamma^{m}(({\gl_2}\mod {p})^{({r+j})})$
by way of $\Gamma^{m}\Gamma^{p^{r+j}}(\gl_2)$. Applying this to our
class in $H^{2mp^{r-1}}(G,\Gamma^{mp^{r+j}}(\gl_2))$
 we hit a class in $H^{2mp^{r-1}}(G,\Gamma^{m}(({\gl_2}\mod p)^{({r+j})}))$, 
which is where $c_r[m]^{(j)}$ of \cite[4.6]{vdK coh}
lives.
On the root subgroup  ${U_\alpha}\mod p$ it is given by the cocycle 
$\Phi(X^{p^{j}},Y^{p^{j}})^{\cup mp^{r-1}}{e_\alpha^{({r+j})[m]}}\mod p$, where ${e_\alpha^{({r+j})[m]}}\mod p$
is our notation for the obvious basis vector of the lowest weight space of $\Gamma^{m}(({\gl_2}\mod p)^{({r+j})})$.
This cocycle is the same as the one used in \cite[4.6]{vdK coh} to construct $c_r[m]^{(j)}$. 
But then their cohomology classes agree on $B$ and $G$ also. 
So we have lifted the $c_r[m]^{(j)}$ of \cite[4.6]{vdK coh}
to a cohomology group with a coefficient module $\Gamma^{m}\Gamma^{p^{r+j}}(\gl_2)$ that is flat over the integers. 

\begin{Notation}
 Simply write $c_r[m]^{(j)}$ for the lift in $H^{2mp^{r-1}}(G,\Gamma^{m}\Gamma^{p^{r+j}}(\gl_2))$.
\end{Notation}

\subsection{Pairings}
In \cite[4.7]{vdK coh} we used the pairing between the modules $\Gamma^{m}({\gl_2}\mod p)^{({r})}$ and 
$S^m({\gl_2^\#}\mod p)^{(r)}$. We want to lift it to a pairing between representations $\Gamma^{m}(X_{r})$ and 
$S^m(Y_{r})$ of $G$
over $\Z$. We take $X=X_{r}=\Gamma^{p^{r}}(\gl_2)$ and define $K=\ker (X\to ({\gl_2}\mod p)^{(r)})$.

Put $Y=Y_{r}=\ker(\Hom_\Z(X,\Z)\to \Hom_\Z(K,\Z/p\Z))$. Then $Y\to \Hom_\Z((X/K),\Z/p\Z)$ is surjective because 
$X$ is a free $\Z$-module. Notice that $\Hom_\Z((X/K),\Z/p\Z)$ is just $({\gl_2^\#}\mod p)^{(r)}$. 
Thus $Y_{r}$ is flat and maps onto $({\gl_2^\#}\mod p)^{(r)}$.

We have a commutative diagram
$$\begin{array}{ccc}
\Gamma^m X\otimes S^mY & \phantom\longrightarrow\llap{$
\relbar\joinrel\relbar\joinrel\relbar\joinrel\relbar\joinrel\relbar\joinrel\relbar\joinrel\longrightarrow$} &\Z\\[.3em]
 \downarrow& & \downarrow\\[.3em]
\Gamma^m (({\gl_2}\mod p)^{(r)})\otimes S^m({\gl_2^\#}\mod p)^{(r)}) & \longrightarrow& \Z/p\Z\\
\end{array}$$
and the left vertical arrow is surjective.
So we have found our lift of the pairing from \cite[4.7]{vdK coh}.

\begin{Remark}
Notice that we do not use the precise shape of $X$ here. 
What matters is that $X$ is free over $\Z$, with
a surjection of $G$ modules $X\to ({\gl_2}\mod p)^{(r)}$,
and that, for $1\leq i\leq r$, we have an element 
in
$H^{2mp^{i-1}}(G,\Gamma^{m}X)$, suggestively denoted $c_i[m]^{(r-i)}$,
 that is mapped to the $c_i[m]^{(r-i)}$ of \cite{vdK coh} 
under the map
induced by $X\to ({\gl_2}\mod p)^{(r)}$. 
\end{Remark}

\subsection{Noetherian base ring}
From now on let $\k$ be an arbitrary commutative noetherian ring.
By base change to $\k$ we get a group scheme over $\k$ that we write again as $G=\SL_2$.
We simply write $X_{r}$ for $X_{r}\otimes_\Z \k$ and we write $Y_{r}$ for $Y_{r}\otimes_\Z \k$. 
We keep suppressing the base ring $\k$ in most notations, so that $X_{r}=\Gamma^{p^{r}}(\gl_2)$,  with
classes  $c_i[m]^{(r-i)}$ in $H^{2mp^{i-1}}(G,\Gamma^{m}X_{r})$.
The commutative diagram above becomes after base change
$$\begin{array}{ccc}
\Gamma^m X_{r}\otimes S^mY_{r} & \phantom\longrightarrow\llap{$
\relbar\joinrel\relbar\joinrel\relbar\joinrel\relbar\joinrel\relbar\joinrel\relbar\joinrel\longrightarrow$} &\k\\[.3em]
 \downarrow& & \downarrow\\[.3em]
\Gamma^m (({\gl_2}\mod p)^{(r)})\otimes S^m(({\gl_2^\#}\mod p)^{(r)})) & \longrightarrow& \k\mod p\\
\end{array}$$

\begin{Lemma}
 If $V$ is a representations of $G$ and $v\in V$, then the subrepresentation generated by $v$ exists and is finitely
generated as a $k$-module.
\end{Lemma}

\begin{proof}
 As $\k[G]$ is a free $\k$-module, this follows from \cite[Expos\'e VI, Lemme 11.8]{SGA3}.
\end{proof}

\subsection{Cup products from pairings}
Let $U$, $V$, $W$, $Z$ be $G$-modules, and $\phi:U\otimes V\to Z$ a $G$-module map. We call $\phi$ a pairing.
Computing with Hochschild complexes one gets cup products $H^i(G,U)\otimes H^j(G,V\otimes W)\to H^{i+j}(G,Z\otimes W)$ induced by $\phi$.
Note that we are not assuming that the modules are flat over $\k$. 
 We think of the Hochschild complex for computing $H^i(G,M)$ as 
$(C^*(G,\k[G])\otimes M)^G$, where $C^*(G,\k[G])$ has a differential graded algebra structure as described
in \cite[section 6.3]{TvdK}.

\subsection{Hitting invariant classes}
\begin{Definition}
 Recall that we call a homomorphism of  $\k$-algebras $f:A\to B$ \emph{noetherian}
 if $f$ makes $B$ into a noetherian 
left $A$-module. 
It is called \emph{power surjective} \cite[Definition 2.1]{FvdK}
if for every $b\in B$ there is an $n\geq1$ so that the power $b^n$ is in the image
of $f$.
\end{Definition}

See \cite[Section 6.2]{TvdK} for some relevant properties of noetherian maps in
cohomology. We are now going to look for noetherian maps. 
We keep the prime $p$ fixed.
Let $r\geq1$. Let $\bar G$ denote $G$ base changed to $(\k\mod p)$, and let $\bar G_r$ denote its $r$-th Frobenius kernel.
(Here we use that $\bar G$ is defined over $\Z/p\Z$.)
We use bars to indicate structures having $(\k\mod p)$ as base ring.
Let $\bar C$ be a finitely generated commutative $(\k\mod p)$-algebra with $\bar G$ action on which $\bar G_r$ acts trivially.
By \cite[Remark 52]{FvdK} we may view $\bar C$ also as an algebra with $G$ action.
Let $\mathcal C$ be a finitely generated commutative $\k$-algebra with $G$ action and let $\pi:\mathcal C\to\bar C$ 
be a power surjective 
equivariant homomorphism.

\begin{Theorem}\label{dominate}
 $H^\even(G,\mathcal C)\to H^0(G,H^*(\bar G_r,\bar C))$ is  noetherian.
\end{Theorem}
\begin{proof}
 By \cite[Thm 1.5, Remark 1.5.1]{FS}
$H^*(\bar G_r,\bar C)$ is a noetherian module over the finitely generated graded
algebra
$$\bar R=\bigotimes_{a=1}^rS^*((\bar{\gl}_2^{(r)})^\#(2p^{a-1}))\otimes
\bar C.$$
Here $(\bar{\gl}_2^{(r)})^\#(2p^{a-1})$ means that one places a copy of $(\bar{\gl}_2^{(r)})^\#$ in degree $2p^{a-1}$.
It is easy to see that the obvious map from 
$\mathcal R=\bigotimes_{a=1}^rS^*(Y_{r}(2p^{a-1}))\otimes
\mathcal C$ to $\bar R$ is noetherian. So 
by invariant theory \cite[Thm.~9]{FvdK},
$H^0(G,H^*(\bar G_r,\bar C))$ is a noetherian module over the finitely
generated algebra $H^0(G,\mathcal R)$. By \cite[Remark 6.7]{TvdK} it now suffices to factor the map 
$H^0(G,\mathcal R)\to H^0(G,H^*(\bar G_r,\bar C))$ as a set map through $H^\even(G,\mathcal C)\to H^0(G,H^*(\bar G_r,\bar C))$.

On a summand $$H^0(G,\bigotimes_{a=1}^rS^{i_a}(Y_{r}(2p^{a-1}))\otimes
\mathcal C)$$ of $H^0(G,\mathcal R)$ we simply take cup product with the (lifted) $c_a[i_a]^{(r-a)}$ according to the pairing of $S^{i_a}(Y_{r})$
with $\Gamma^{i_a}(X_r)=\Gamma^{i_a}\Gamma^{p^{r}}(\gl_2)$. In the proof of
\cite[Cor.~4.8]{vdK coh} one has a similar description of 
the map to $H^*(\bar G_r,\bar C)$ on the summand 
$$H^0(G,\bigotimes_{a=1}^rS^{i_a}((\bar{\gl}_2^{(r)})^\#(2p^{a-1}))\otimes
\bar C)$$ of $H^0(G,R)$. The required factoring as a set map thus follows from the compatibility of the pairings
and the fact that the lifted $c_a[i_a]^{(r-a)}$ are lifts of their mod $p$ namesakes.
\end{proof}

Recall that $G$ is the group scheme $\SL_2$ over the noetherian base ring $\k$.
Now let $A$ be a finitely generated commutative $\k$-algebra with $G$ action.
\begin{Theorem}[CFG in rank one]
 $H^*(G,A)$ is a finitely generated algebra.
\end{Theorem}
\begin{proof}
 Recall that $A$ comes with an increasing 
filtration $A_{\leq0}\subseteq A_{\leq1}\cdots$ where $A_{\leq i}$ denotes the largest $G$-submodule all whose weights 
$\lambda$  satisfy $\hgt \lambda=\sum_{\beta>0}\langle \beta,\beta^\vee\rangle\leq i$.
(Actually there is now only one positive root, so that the sum has just one term.)
The associated graded algebra is the Grosshans graded ring $\gr A$. Let $\mathcal A$ be the Rees ring of the filtration.
So $\mathcal A$ is
the subring of the polynomial ring $A[t]$ generated by the subsets
$t^iA_{\leq i}$. 
Let $\bar A=A/pA$. As in \cite[Section 3]{vdK coh} we choose $r$ so big that
 $x^{p^r}\in \gr \bar A$ for all $x\in \hull_\nabla(\gr \bar A)$.
Put $\bar C=(\gr\bar A)^{\bar G_r}$. By \cite[Thm.~30]{FvdK} the algebra $\mathcal A/t\mathcal A=\gr A$ is finitely
generated,  so $\mathcal A$ is finitely generated.  By \cite[Thm.~35]{FvdK} the map $\gr A\to \gr \bar A$ is power surjective.
Then so is the map $\mathcal A\to \gr \bar A$, because $\mathcal A\to \gr A$ is surjective. 
Now take a finitely generated $G$ invariant subalgebra $\mathcal C$ of
the inverse image of $\bar C$ in
$\mathcal A$ in such a way that $\mathcal C\to \bar C$ is power surjective.
By theorem \ref{dominate} the map $H^\even(G,\mathcal C)\to H^0(G,H^*(\bar G_r,\bar C))$ is  noetherian.
By \cite[Theorem 1.5, Remark 1.5.1]{FS} the $H^*(\bar G_r,\bar C)$-module $H^*(\bar G_r,\gr\bar A)$ is noetherian
and by \cite[Theorems 9, 12]{FvdK} it follows that $H^0(G,H^*(\bar G_r,\bar C))\to H^0(G,H^*(\bar G_r,\gr\bar A))$
 is noetherian.
Then so is $H^\even(G,\mathcal C)\to H^0(G,H^*(\bar G_r,\gr\bar A))$, hence also
$H^\even(G,\mathcal A)\to H^0(G,H^*(\bar G_r,\gr\bar A))$.
This is what is needed to argue as in \cite[4.10]{vdK coh}
that $H^\even(G,\mathcal A)\to H^*(G,\gr\bar A)$ is noetherian. And then one concludes as in \cite{vdK coh}
that $H^\even(G,\mathcal A)\to H^*(G,\bar A)$ is noetherian. 
But $\mathcal A\to \bar A$ factors through $A$.
It follows that $H^\even(G, A)\to H^*(G,\bar A)$ is noetherian. 
As $p$ was an arbitrary prime, \cite[Thm.~49]{FvdK}, or rather the last item of
\cite[Theorem 10.5]{goodfamily}, applies.
\end{proof}


\end{document}